\documentclass[12pt]{article}
\usepackage[utf8]{inputenc}
\usepackage{amsmath}
\usepackage[hmargin=1in,vmargin=1in]{geometry}
\usepackage{graphics}
\usepackage[toc,page]{appendix}
\usepackage{amsfonts}
\usepackage{MnSymbol}
\usepackage{wasysym}
\usepackage{enumerate}
\usepackage{amsthm}
\usepackage{graphicx}
\usepackage{float}
\usepackage{chngcntr}
\theoremstyle{plain}
\newtheorem{thm}{Theorem}[section]
\newtheorem{lemma}[thm]{Lemma}

\newtheorem{prop}[thm]{Proposition}

\theoremstyle{definition}

\numberwithin{equation}{section}
\counterwithin{figure}{section}
\title{Rudnick and Soundarajan's Theorem over Prime Polynomials for the Rational Function Field}
\author{J. MacMillan}
\date{}
\newcommand\blfootnote[1]{%
  \begingroup
  \renewcommand\thefootnote{}\footnote{#1}%
  \addtocounter{footnote}{-1}%
  \endgroup
}
\begin{document}
\maketitle
\blfootnote{Date: Octobber 13, 2020}
\par\noindent
ABSTRACT: In this paper, we use the methods of Andrade, Rudnick and Soundarajan to prove a Theorem about Lower bounds of moments of quadratic Dirichlet L-functions associated to monic irreducible polynomials over function fields.
\section{Introduction}
A fundamental problem in Analytic Number Theory is to understand the asymptotic behaviour of moments of families of L-functions. For example, in the case of the Riemann-zeta function, a problem is to establish an asymptotic formula for
\begin{equation}\label{eq:1.1}
    M_k(T)=\int_1^T\left|\zeta\left(\frac{1}{2}+it\right)\right|^{2k}dt,
\end{equation}
where $k$ is a positive integer and $T\rightarrow\infty$. Asymptotic formulas for the first two moments have been explicitly calculated, the first by Hardy and Littlewood \cite{Hardy1916}, in which they proved that 
\begin{equation*}
    M_1(T)\sim T\log T
\end{equation*}
and Ingham \cite{Ingham1926} proved that
\begin{equation*}
    M_2(T)\sim \frac{1}{2\pi^2}T\log^4T.
\end{equation*}
Although no higher moments have been explicitly calculated , it is conjectured that 
\begin{equation*}
    M_k(T)\sim c_kT(\log T)^{k^2}
\end{equation*}
where, due to Conrey and Ghosh \cite{ConreyGhosh1992}, the constant $c_k$ assumes a more explicit form, namely
\begin{equation*}
    c_k=\frac{a_kg_k}{\Gamma(k^2+1)},
\end{equation*}
where
\begin{equation*}
    a_k=\prod_P\left[\left(1-\frac{1}{p^2}\right)^{k^2}\sum_{m\geq 0}\frac{d_k(m)^2}{p^m}\right],
\end{equation*}
$g_k$ is an integer when $k$ is an integer and $d_k(m)$ is the number of ways to represent $m$ as a product of $k$ factors. Ramachandra \cite{Ramachandra1980} obtained a lower bound for moments of the Riemann-zeta function for positive integers $k$. In particular he showed that  
\begin{equation*}
    M_k(T)\gg T(\log T)^{k^2}.
\end{equation*}
For the family of Dirichlet L-functions, $L(s,\chi_d)$ associated to the quadratic character $\chi_d$, a problem is to understand the asymptotic behaviour of
\begin{equation}\label{eq:1.2}
    S_k(X)=\sum_{|d|\leq X}L\left(\frac{1}{2},\chi_d\right)^k,
\end{equation}
where the sum is over fundamental discriminants $d$ as $X\rightarrow\infty$. Jutila \cite{Jutila1981}, proved that
\begin{equation*}
    S_1(X)\sim C_1X\log X
\end{equation*}
and 
\begin{equation*}
    S_2(X)\sim C_2X\log^3X,
\end{equation*}
where $C_1$ and $C_2$ are positive constants. Restricting $d$ to be odd, square-free and positive, so that $\chi_{8d}$ are real, primitive characters with conductor $8d$ and with $\chi_{8d}(-1)=1$, Soundarajan \cite{Soundararajan2000NonvanishingS=1/2} proved 
\begin{equation*}
    \sum_{|d|\leq X}L\left(\frac{1}{2},\chi_{8d}\right)^3\sim C_3X\log^6X,
\end{equation*}
for some positive constant $C_3$. It is conjectured, by Keating and Snaith \cite{Keating2000} that
\begin{equation*}
    S_k(X)\sim C_k X(\log X)^{\frac{k(k+1)}{2}},
\end{equation*}
for some positive constant $C_k$. Rudnick and Soundarajan \cite{Rudnick2006} proved a result for the lower bounds for moments of these Dirichlet L-functons, when $k$ is an even integer. In particular, they showed that
\begin{equation*}
    S_k(X)\gg X(\log X)^{\frac{k(k+1)}{2}}.
\end{equation*}
\par\noindent
In the Function Field setting, the analogue problem of (\ref{eq:1.2}) is to understand the asymptotic behaviour of 
\begin{equation}\label{eq:1.3}
    I_k(g)=\sum_{D\in\mathcal{H}_{2g+1}}L\left(\frac{1}{2},\chi_D\right)^k
\end{equation}
as $|D|=q^{\text{deg}(D)}\rightarrow\infty$, where $\mathcal{H}_{2g+1}$ denotes the space of monic, square-free polynomials of degree $2g+1$ over $\mathbb{F}_q[T]$. In the setting of fixing $q$, where in particular $q\equiv 1(\text{mod }4)$, and letting $g\rightarrow\infty$, Andrade and Keating \cite{Andrade2012} and Florea \cite{Florea2017} proved that
\begin{equation*}
    I_1(g)\sim \tilde{C_1}|D|\log_q|D|
\end{equation*}
for some positive constant $\tilde{C_1}$. Andrade and Keating \cite{Andrade2014} conjectured that
\begin{equation*}
    I_k(g)\sim \tilde{C_k}|D|(\log_q|D|)^{\frac{k(k+1)}{2}}
\end{equation*}
for some positive constant $\tilde{C_k}$. The second, third and fourth moments were explicitly calculated by Florea \cite{Florea2017a,Florea2017c}, where the computed constants agreed with the constants conjectured by Andrade and Keating. Similar to the calculations done by Rudnick and Soundarajan in the number field setting, Andrade \cite{Andrade2016a} proved a result for the lower bounds of (\ref{eq:1.3}) when $k$ is an even integer. In particular, he proved that
\begin{equation*}
    I_k(g)\gg |D|(\log_q|D|)^{\frac{k(k+1)}{2}}.
\end{equation*}
\par\noindent
Another problem in Function Fields is to understand the asymptotic behaviour of 
\begin{equation}\label{eq:1.4}
    \sum_{P\in\mathcal{P}_{2g+1}}L\left(\frac{1}{2},\chi_P\right)^k
\end{equation}
where $\mathcal{P}_{2g+1}$ denotes the spaces of monic, irreducible polynomials of degree $2g+1$ over $\mathbb{F}_q[T]$. Andrade and Keating \cite{Andrade2013} proved
\begin{equation*}
    \sum_{P\in\mathcal{P}_{2g+1}}(\log_q|P|)L\left(\frac{1}{2},\chi_P\right)\sim |P|\log_q|P|,
\end{equation*}
while, along with Bui and Florea \cite{Bui2020}, in the same paper computed the second moment
\begin{equation*}
    \sum_{P\in\mathcal{P}_{2g+1}}L\left(\frac{1}{2},\chi_P\right)^2\sim|P|(\log_q|P|)^2.
\end{equation*}
In this paper, we use similar methods to that of Andrade, Rudnick and Soundarajan to obtain a lower bound for (\ref{eq:1.4}). The main result is the following.
\begin{thm}
For every even natural number $k$ and $n=2g+1$ or $n=2g+2$ we have,
\begin{equation*}
    \frac{1}{|\mathcal{P}_{n}|}\sum_{P\in\mathcal{P}_{n}}L\left(\frac{1}{2},\chi_P\right)^k \gg_k (\log_q|P|)^{\frac{k(k+1)}{2}}.
\end{equation*}
\end{thm}
\section{Background and Preliminaries}
Before we prove Theorem 1.1, we state some facts which can generally be found in  \cite{Rosen2002}. Fix a finite field $\mathbb{F}_q$, where $q\equiv 1(\text{mod }4)$ and let $\mathbb{A}=\mathbb{F}_q[T]$ be the polynomial ring over $\mathbb{F}_q$ and let $k=\mathbb{F}_q(T)$ be the rational function field over $\mathbb{F}_q$. 
We use the notion $\mathbb{A}^+, \mathbb{A}^+_n$ and $\mathbb{A}^+_{\leq n}$ to denote the set of all monic polynomials in $\mathbb{A}$, the set of all monic polynomials of degree $n$ in $\mathbb{A}$ and the set of all monic polynomials of degree less than or equal to $n$ in $\mathbb{A}$ respectively. The zeta function associated to $\mathbb{A}$ is defined by the infinite series
\begin{equation*}
    \zeta_{\mathbb{A}}(s)=\sum_{f\in\mathbb{A}^+}\frac{1}{|f|^s},
\end{equation*}
where $|f|=q^{\text{deg}(f)}$ if $f\neq 0$ and $|f|=0$ if $f=0$. There are $q^n$ monic polynomials of degree $n$, therefore we have
\begin{equation*}
    \zeta_{\mathbb{A}}(s)=\frac{1}{1-q^{1-s}}.
\end{equation*}
Let $P$ be a monic irreducible polynomial in $\mathbb{A}$ of odd degree. We denote by $\chi_P$ the quadratic character defined in terms of the quadratic residue symbol for $\mathbb{F}_q[T]$:
\begin{equation*}
    \chi_P(f)=\left(\frac{P}{f}\right)
\end{equation*}
where $f\in\mathbb{A}$. For more details see \cite{Rosen2002}, chapter 3. Thus, if $Q\in\mathbb{A}$ is a monic, irreducible polynomial, we have
\[
\chi_P(Q)=
\begin{cases}
0 &\text{if } Q|P\\
1 &\text{if } Q\nmid P \text{ and } P \text{ is a square mod }P\\
-1 &\text{if } Q\nmid P \text{ and } P \text{ is not a square mod }P.
\end{cases}
\]
Thus, the corresponding Dirichlet L-function associated with the quadratic character $\chi_P$ is defined as
\begin{equation*}
    L(s,\chi_P)=\sum_{f\in\mathbb{A}^+}\frac{\chi_P(f)}{|f|^s}.
\end{equation*}
Let 
\begin{equation*}
    \mathcal{P}_{n}=\{P\in\mathbb{A},\text{ monic, irreducible, deg}(P)=n\}.
\end{equation*}
If $P\in\mathcal{P}_{2g+1}$, then by the argument given in \cite{Andrade2013}, $L(s,\chi_P)$ is a polynomial in $q^{-s}$ of degree $2g$ given by
\begin{equation*}
    L(s,\chi_P)=\sum_{n=0}^{2g}\sum_{f\in\mathbb{A}^+_n}\chi_P(f)q^{-ns}
\end{equation*}
and satisfies the functional equation
\begin{equation*}
    L(s,\chi_P)=(q^{1-2s})^gL(1-s,\chi_P)
\end{equation*}
and the Riemann Hypothesis for curves proved by Weil \cite{Weil1948} tells us that all the zeros of $L(s,\chi_P)$ have real part equal to $\frac{1}{2}$. The next results will be fundamental in proving Theorem 1.1.
\begin{thm}[Prime Polynomial Theorem]
We have that
\begin{equation*}
    |\mathcal{P}_n|=\frac{q^n}{n}+O\left(\frac{q^{\frac{n}{2}}}{n}\right).
\end{equation*}
\begin{proof}
See \cite{Rosen2002}, Theorem 2.2.
\end{proof}
\begin{prop}
If $f\in\mathbb{A}$, monic, deg$(f)>0$ and $f$ is not a perfect square, then 
\begin{equation*}
    \left|\sum_{P\in\mathcal{P}_n}\chi_P(f)\right|\ll \frac{\text{deg}(f)}{n}q^{\frac{n}{2}}.
\end{equation*}
\end{prop}
\begin{proof}
See \cite{Rudnick2008}, section 2.
\end{proof}
\end{thm}
\begin{lemma}[Approximate Functional Equation]
For $P\in\mathcal{P}_{2g+1}$, we have
\begin{equation}
    L\left(\frac{1}{2},\chi_P\right)=\sum_{f_1\in\mathbb{A}^+_{\leq g}}\frac{\chi_P(f_1)}{\sqrt{|f_1|}}+\sum_{f_2\in\mathbb{A}^+_{\leq g-1}}\frac{\chi_P(f_2)}{\sqrt{|f_2|}}.
\end{equation}
\end{lemma}
\begin{proof}
The proof is similar to that given in \cite{Andrade2012}, Lemma 3.3. 
\end{proof}
\section{Proof of Theorem 1.1}
In this section, we prove Theorem 1.1. Let $k$ be a given even number and let $x=\frac{2(2g)}{15k}$. We define
\begin{equation}
    A(P)=\sum_{n\in\mathbb{A}^+_{\leq x}}\frac{\chi_P(n)}{\sqrt{|n|}}
\end{equation}
and let
\begin{equation}
    S_1=\sum_{P\in\mathcal{P}_{2g+1}}L\left(\frac{1}{2},\chi_P\right)A(P)^{k-1}
\end{equation}
and 
\begin{equation}
    S_2=\sum_{P\in\mathcal{P}_{2g+1}}A(P)^k.
\end{equation}
An application of Triangle Inequality followed by H\"older's inequality gives us that 
\begin{align*}
   & \left|\sum_{P\in\mathcal{P}_{2g+1}}L\left(\frac{1}{2},\chi_P\right)A(P)^{k-1}\right|\leq \sum_{P\in\mathcal{P}_{2g+1}}|L(\frac{1}{2},\chi_P)||A(P)|^{k-1}\\
   &\leq \left(\sum_{P\in\mathcal{P}_{2g+1}}L(\frac{1}{2},\chi_P)^k\right)^{\frac{1}{k}}\left(\sum_{P\in\mathcal{P}_{2g+1}}A(P)^k\right)^{\frac{k-1}{k}}.
\end{align*}
Rearranging gives
\begin{align*}
    \sum_{P\in\mathcal{P}_{2g+1}}L\left(\frac{1}{2},\chi_P\right)^k&\geq \frac{\left(\sum_{P\in\mathcal{P}_{2g+1}}L(\frac{1}{2},\chi_P)A(P)^{k-1}\right)^k}{\left(\sum_{P\in\mathcal{P}_{2g+1}}A(P)^k\right)^{k-1}}\\
    &=\frac{S_1^k}{S_2^{k-1}}.
\end{align*}
Thus, to prove Theorem 1.1, we only need to give satisfactory estimates for $S_1$ and $S_2$.
\subsection{Evaluating $S_2$}
We have that 
\begin{equation*}
    A(P)^k=\sum_{\substack{n_j\in\mathbb{A}^+_{\leq x}\\ j=1,\dotsc,k}}\frac{\chi_P(n_1\dotsc n_k)}{\sqrt{|n_1|\dotsc |n_k|}}.
\end{equation*}
So
\begin{align*}
    S_2&=\sum_{\substack{n_j\in\mathbb{A}^+_{\leq x}\\ j=1,\dotsc,k}}\frac{1}{\sqrt{|n_1|\dotsc|n_k|}}\sum_{P\in\mathcal{P}_{2g+1}}\chi_P(n_1\dotsc n_k)\\
    &=\sum_{\substack{n_j\in\mathbb{A}^+_{\leq x}\\ j=1,\dotsc,k\\n_1\dotsc n_k=\square}}\frac{1}{\sqrt{|n_1|\dotsc |n_k|}}\sum_{P\in\mathcal{P}_{2g+1}}1+\sum_{\substack{n_j\in\mathbb{A}^+_{\leq x}\\ j=1,\dotsc,k\\n_1\dotsc n_k=\neq\square}}\frac{1}{\sqrt{|n_1|\dotsc |n_k|}}\sum_{P\in\mathcal{P}_{2g+1}}\chi_P(n_1\dotsc n_k).
\end{align*}
Using Theorem 2.1 and Proposition 2.2, we have 
\begin{align*}
    S_2&=\frac{|P|}{\log_q|P|}\sum_{\substack{n_j\in\mathbb{A}^+_{\leq x}\\ j=1,\dotsc,k\\n_1\dotsc n_k=\square}}\frac{1}{\sqrt{|n_1|\dotsc |n_k|}}+\sum_{\substack{n_j\in\mathbb{A}^+_{\leq x}\\ j=1,\dotsc,k\\n_1\dotsc n_k=\square}}\frac{1}{\sqrt{|n_1|\dotsc |n_k|}}O\left(\frac{|P|^{\frac{1}{2}}}{\log_q|P|}\right)\\&+\sum_{\substack{n_j\in\mathbb{A}^+_{\leq x}\\ j=1,\dotsc,k\\n_1\dotsc n_k\neq\square}}\frac{1}{\sqrt{|n_1|\dotsc |n_k|}}O\left(\frac{|P|^{\frac{1}{2}}}{\log_q|P|}\text{deg}(n_1\dotsc n_k)\right).
\end{align*}
Using the choice of $x$ given and after some manipulation with the $O$-terms, we get that 
\begin{equation*}
    S_2=\frac{|P|}{\log_q|P|}\sum_{\substack{n_j\in\mathbb{A}^+_{\leq x}\\ j=1,\dotsc,k\\n_1\dotsc n_k=\square}}\frac{1}{\sqrt{|n_1|\dotsc |n_k|}}+O(|P|^{\frac{17}{30}}).
\end{equation*}
Writing $n_1\dotsc n_k=m^2$ we see that
\begin{align*}
    \sum_{m\in\mathbb{A}^+_{\leq x}}\frac{d_k(m^2)}{|m|}\leq \sum_{\substack{n_j\in\mathbb{A}^+_{\leq x}\\ j=1,\dotsc,k\\n_1\dotsc n_k=\square}}\frac{1}{\sqrt{|n_1|\dotsc |n_k|}}\leq \sum_{m\in\mathbb{A}^+_{\leq kx}}\frac{d_k(m^2)}{|m|}.
    \end{align*}
    Using similar methods to that given in \cite{Andrade2016a}, we see that
    \begin{equation}\label{eq3}
        \sum_{m\in\mathbb{A}^+_{\leq z}}\frac{d_k(m^2)}{|m|}\sim C(k)z^{\frac{k(k+1)}{2}}.
    \end{equation}
    Therefore we can conclude that 
    \begin{equation}\label{eqs1}
        S_2\ll |P|(\log_q|P|)^{\frac{k(k+1)}{2}-1}.
    \end{equation}
    \subsection{Evaluating $S_1$}
    Using Lemma 2.3, we have that
    \begin{align*}
        S_1&=\sum_{\substack{f_1\in\mathbb{A}^+_{\leq g}\\n_j\in\mathbb{A}^+_{\leq x}\\j=1,\dotsc k-1}}\frac{1}{\sqrt{|f_1||n_1|\dotsc|n_{k-1}|}}\sum_{P\in\mathcal{P}_{2g+1}}\chi_P(f_1n_1\dotsc n_{k-1})\\&+\sum_{\substack{f_2\in\mathbb{A}^+_{\leq g-1}\\n_j\in\mathbb{A}^+_{\leq x}\\j=1,\dotsc k-1}}\frac{1}{\sqrt{|f_2||n_1|\dotsc|n_{k-1}|}}\sum_{P\in\mathcal{P}_{2g+1}}\chi_P(f_2n_1\dotsc n_{k-1}).
    \end{align*}
    The two sums for $S_1$ are the same apart from the size of the sums. So we will only estimate the first sum, as the second follows from replacing $g$ with $g-1$. Thus we have that
    \begin{align*}
        &\sum_{\substack{f\in\mathbb{A}^+_{\leq g}\\n_j\in\mathbb{A}^+_{\leq x}\\j=1,\dotsc, k-1}}\frac{1}{\sqrt{|f||n_1|\dotsc| n_{k-1}|}}\sum_{P\in\mathcal{P}_{2g+1}}\chi_P(fn_1\dotsc n_{k-1})\\
       &= \sum_{\substack{f\in\mathbb{A}^+_{\leq g}\\n_j\in\mathbb{A}^+_{\leq x}\\j=1,\dotsc, k-1\\fn_1\dotsc n_{k-1}=\square}}\frac{1}{\sqrt{|f||n_1|\dotsc |n_{k-1}|}}\sum_{P\in\mathcal{P}_{2g+1}}1+\sum_{\substack{f\in\mathbb{A}^+_{\leq g}\\n_j\in\mathbb{A}^+_{\leq x}\\j=1,\dotsc,k-1\\fn_1\dotsc n_{k-1}\neq \square}}\frac{1}{\sqrt{|f||n_1|\dotsc |n_{k-1}|}}\sum_{P\in\mathcal{P}_{2g+1}}\chi_P(fn_1\dotsc n_{k-1}).
    \end{align*}
    Using Theorem 2.1 and Proposition 2.2, we have that
    \begin{align*}
        &\sum_{\substack{f\in\mathbb{A}^+_{\leq g}\\n_j\in\mathbb{A}^+_{\leq x}\\j=1,\dotsc, k-1}}\frac{1}{\sqrt{|f||n_1|\dotsc |n_{k-1}|}}\sum_{P\in\mathcal{P}_{2g+1}}\chi_P(fn_1\dotsc n_{k-1})\\
        &=\frac{|P|}{\log_q|P|}\sum_{\substack{f\in\mathbb{A}^+_{\leq g}\\n_j\in\mathbb{A}^+_{\leq x}\\j=1,\dotsc,k-1\\fn_1\dotsc n_{k-1}=\square}}\frac{1}{\sqrt{|f||n_1|\dotsc| n_{k-1}|}}+\sum_{\substack{f\in\mathbb{A}^+_{\leq g}\\n_j\in\mathbb{A}^+_{\leq x}\\j=1,\dotsc,k-1\\fn_1\dotsc n_{k-1}=\square}}\frac{1}{\sqrt{|f||n_1|\dotsc |n_{k-1}|}}O\left(\frac{|P|^{\frac{1}{2}}}{\log_q|P|}\right)\\
        &+\sum_{\substack{f\in\mathbb{A}^+_{\leq g}\\n_j\in\mathbb{A}^+_{\leq x}\\j=1,\dotsc,k-1\\fn_1\dotsc n_{k-1}\neq\square}}\frac{1}{\sqrt{|f||n_1|\dotsc| n_{k-1}|}}O\left(\frac{|P|^{\frac{1}{2}}}{\log_q|P|}\text{deg}(fn_1\dotsc n_{k-1})\right)
    \end{align*}
    Using the choice of $x$ given and after some manipulation with the O-terms, we get that
    \begin{align}\label{eq:3.4}
        &\sum_{\substack{f\in\mathbb{A}^+_{\leq g}\\n_j\in\mathbb{A}^+_{\leq x}\\j=1,\dotsc, k-1}}\frac{1}{\sqrt{|f||n_1|\dotsc| n_{k-1}|}}\sum_{P\in\mathcal{P}_{2g+1}}\chi_P(fn_1\dotsc n_{k-1})\nonumber\\
        &=\frac{|P|}{\log_q|P|}\sum_{\substack{f\in\mathbb{A}^+_{\leq g}\\n_j\in\mathbb{A}^+_{\leq x}\\j=1,\dotsc,k-1\\fn_1\dotsc n_{k-1}=\square}}\frac{1}{\sqrt{|f||n_1|\dotsc |n_{k-1}|}}+O(|P|^{\frac{49}{60}})
    \end{align}
    For the main term we let $n_1\dotsc n_{k-1}=rh^2$ and $f=rl^2$ and thus we get
    \begin{equation*}
        \sum_{\substack{f\in\mathbb{A}^+_{\leq g}\\n_j\in\mathbb{A}^+_{\leq x}\\j=1,\dotsc,k-1\\fn_1\dotsc n_{k-1}=\square}}\frac{1}{\sqrt{|f||n_1|\dotsc |n_{k-1}|}}=\sum_{\substack{n_j\in\mathbb{A}^+_{\leq x}\\j=1,\dotsc,k-1\\n_1\dotsc n_{k-1}=rh^2}}\frac{1}{|rh|}\sum_{l\in\mathbb{A}^+_{\leq \frac{g-\text{deg}(r)}{2}}}\frac{1}{|l|}.
    \end{equation*}
    We see that
    \begin{equation*}
        \sum_{l\in\mathbb{A}^+_{\leq \frac{g-\text{deg}(r)}{2}}}\frac{1}{|l|}\sim C(r,h)(\log_q|P|)
    \end{equation*}
    for some positive constant $C(r,h)$. Therefore it follows that the main term in (\ref{eq:3.4}) is
    \begin{align*}
        &\gg \frac{|P|}{\log_q|P|}\sum_{\substack{n_j\in\mathbb{A}^+_{\leq x}\\j=1,\dotsc,n_{k-1}\\n_1\dotsc n_{k-1}=rh^2}}\frac{1}{|rh|}\gg \frac{|P|}{\log_q|P|}\sum_{\substack{r,h \text{ monic}\\\text{deg}(rh^2)\leq x}}\frac{d_{k-1}(rh^2)}{|rh|}\gg |P|(\log_q|P|)^{\frac{k(k+1)}{2}-1}. 
    \end{align*}
    where the last bound follows from (\ref{eq3}) but replacing $k$ with $k-1$. Therefore we have that
    \begin{equation}\label{eqs2}
        S_1\gg |P|(\log_q|P|)^{\frac{k(k+1)}{2}-1}.
    \end{equation}
    Combining (\ref{eqs1}) and (\ref{eqs2}) and using Theorem 2.1 proves Theorem 1.1.\\
    \par\noindent
    \textbf{Acknowledgement:} The author is grateful to the Leverhulme Trust (RPG-2017-320) for the support given during this research through a PhD studentship. The author would also like to thank Dr. Julio Andrade for suggesting this problem to me, as well as his useful advice during the course of the research.
\bibliographystyle{plain}
\bibliography{lowerboundprime}

\begin{thebibliography}{10}

\bibitem{Andrade2016a}
J.C. Andrade.
\newblock {Rudnick and Soundararajan's theorem for function fields}.
\newblock {\em Finite Fields Appl.}, 37:311--327, 2016.

\bibitem{Andrade2012}
J.C. Andrade and J.P. Keating.
\newblock {The mean value of L(1/2, $\chi$) in the hyperelliptic ensemble}.
\newblock {\em J. Number Theory}, 132:2793--2816, 2012.

\bibitem{Andrade2013}
J.C. Andrade and J.P. Keating.
\newblock {Mean value theorems for L-functions over prime polynomials for the
  rational function field}.
\newblock {\em Acta Arith.}, 161(4):371--385, 2013.

\bibitem{Andrade2014}
J.C. Andrade and J.P. Keating.
\newblock {Conjectures for the integral moments and ratios of L-functions over
  function fields}.
\newblock {\em J. Number Theory}, 142:102--148, 2014.

\bibitem{Bui2020}
H.~Bui and A.~Florea.
\newblock {Moments of Dirichlet L–functions with prime conductors over
  function fields}.
\newblock {\em Finite Fields their Appl.}, 64:1--21, 2020.

\bibitem{ConreyGhosh1992}
B.~Conrey and A.~Ghosh.
\newblock {Mean-Value Theorems in the theory of the Riemann-zeta function}.
\newblock {\em Proceedings of the Amalfi Conference on Analytic Number Theory
  (Maiori,1989)}, pages 35--39, 1992.

\bibitem{Florea2017}
A.~Florea.
\newblock {Improving the Error Term in the Mean Value of in the Hyperelliptic
  Ensemble}.
\newblock {\em Int. Math. Res. Not. IMRN}, 20:6119--6148, 2017.

\bibitem{Florea2017a}
A.~Florea.
\newblock {The fourth moment of quadratic Dirichlet L-functions over function
  fields}.
\newblock {\em Geom. Funct. Anal.}, 27(3):541--595, 2017.

\bibitem{Florea2017c}
A.~Florea.
\newblock {The second and third moment of $L(1/2,\chi)$ in the hyperelliptic
  ensemble}.
\newblock {\em Forum Math.}, 29(4):873--892, 2017.

\bibitem{Hardy1916}
G.H. Hardy and J.E. Littlewood.
\newblock {Contributions to the theory of the riemann zeta-function and the
  theory of the distribution of primes}.
\newblock {\em Acta Math.}, 41(1):119--196, 1916.

\bibitem{Ingham1926}
A.E Ingham.
\newblock {Mean-Value Theorems in the theory of the Riemann-zeta function}.
\newblock {\em Proc. Lond. Math. Soc}, 27(1621):273--300, 1926.

\bibitem{Jutila1981}
M.~Jutila.
\newblock {On the Mean Value of $L(1/2,\chi)$ for Real Characters}.
\newblock {\em Analysis}, 1981.

\bibitem{Keating2000}
J.P Keating and N.C Snaith.
\newblock {Random Matrix Theory and L-Functions at s = 1/2}.
\newblock {\em Commun. Math. Phys.}, 214(1):91--100, 2000.

\bibitem{Ramachandra1980}
K~Ramachandra.
\newblock {Some remarks on the mean value of the Riemann-zeta function and
  other Dirichlet Series II}.
\newblock {\em Hardy-Ramanujan J.}, 3:1--25, 1980.

\bibitem{Rosen2002}
M.~Rosen.
\newblock {\em {Number Theory in Function Fields, Graduate Texts in Matematics,
  Vol. 210}}.
\newblock Springer-Verlag, New York, 2002.

\bibitem{Rudnick2008}
Z.~Rudnick.
\newblock {Traces of high powers of the Frobenius class in the hyperelliptic
  ensemble}.
\newblock {\em Acta Arith.}, 143:81--99, 2008.

\bibitem{Rudnick2006}
Z.~Rudnick and K.~Soundararajan.
\newblock {Lower bounds for moments of L-functions: Sympletic and Orthogonal
  Examples. Multiplke Dirichlet series, automorphic forms and analytic number
  theory}.
\newblock {\em Proc. Symp. Pure Math.}, 75:293--303, 2006.

\bibitem{Soundararajan2000NonvanishingS=1/2}
K.~Soundararajan.
\newblock {Nonvanishing of quadratic Dirichlet L-functions at $s=1/2$}.
\newblock {\em Ann. of Math.}, 152(2):447--488, 2000.

\bibitem{Weil1948}
A.~Weil.
\newblock {Sur les Courbes Alg{\'e}briques et les Vari{\'e}t{\'e}s qui s'en
  D{\'e}duisent}.
\newblock {\em Hermann, Paris}, 1948.

\end{thebibliography}
Department of Mathematics, University of Exeter, Exeter, EX4 4QF, UK\\
\textit{E-mail Address:}jm1015@exeter.ac.uk
\end{document}